\documentclass[11pt,leqno]{article}
\usepackage{a4wide} 
\usepackage[english]{babel}

\usepackage[utf8]{inputenc}
\usepackage{amsmath}
\usepackage{amssymb}
\usepackage{amsthm}
\usepackage{amsfonts}
\usepackage{amsmath}
\usepackage{graphicx}
\usepackage{listings}
\usepackage[table]{xcolor}
\usepackage{listings}
\usepackage{url}
\usepackage{enumerate}
\usepackage{blindtext}
\usepackage{hyperref}
\usepackage{xcolor}
\usepackage{subcaption}
\usepackage{tikz}
\usepackage{bbm}
\usepackage{dsfont}
\usepackage{mathtools}
\usetikzlibrary{arrows}
\everymath{\displaystyle}

\hypersetup{
    colorlinks=true,
    linkcolor=blue,
    filecolor=magenta,
    urlcolor=cyan,
    pdftitle={overleaf Example},
}

\renewcommand{\Im}[1]{{\mathfrak{Im}}(#1)}

\theoremstyle{definition}
\newtheorem{theorem}{Theorem}

\newtheorem{corollary}{Corollary}
\newtheorem{example}{Example}
\newtheorem{remark}{Remark}
\newtheorem{lemma}{Lemma}
\newtheorem*{definition}{Definition}

\DeclareMathOperator{\Fix}{fix}
\DeclareMathOperator{\lcm}{lcm}

\newcommand{\rot}{r}

\definecolor{codegreen}{rgb}{0,0.6,0}
\definecolor{codegray}{rgb}{0.5,0.5,0.5}
\definecolor{codepurple}{rgb}{0.58,0,0.82}
\definecolor{backcolour}{rgb}{0.95,0.95,0.92}

\lstdefinestyle{mystyle}{
    backgroundcolor=\color{backcolour},
    commentstyle=\color{codegreen},
    keywordstyle=\color{magenta},
    numberstyle=\tiny\color{codegray},
    stringstyle=\color{codepurple},
    basicstyle=\ttfamily\footnotesize,
    breakatwhitespace=false,
    breaklines=true,
    captionpos=b,
    keepspaces=true,
    numbers=left,
    numbersep=5pt,
    showspaces=false,
    showstringspaces=false,
    showtabs=false,
    tabsize=2
}

\def\inn{\subseteq}
\newcommand\gen[1]{\left\langle #1\right\rangle}
\newcommand\paren[1]{\left( #1\right)}
\newcommand\set[1]{\left\{ #1\right\}}
\newcommand\abs[1]{\left| #1\right|}
\newcommand\ideal[1]{\abs{\frac{\Sigma[X]}{(\Lambda, X^{#1} - 1)}}}

\newcommand\GCD[2]{\left(#1 : #2\right)}

\newcommand\Gast[2]{G_{#1,#2}}
\newcommand\Vast[2]{V_{#1,#2}}
\newcommand\East[2]{E_{#1,#2}}

\newcommand\Gnk[0]{\Gast{n}{k}}
\newcommand\Vnk[0]{\Vast{n}{k}}
\newcommand\Enk[0]{\East{n}{k}}

\newcommand\Fac[2]{F_{#2}(#1)}
\newcommand\Act[2]{A_{#2}(#1)}

\newcommand\Substr[3]{#1[#2..#3)}
\newcommand\One[1]{\mathbbm{1}_{#1}}

\def\Polu{U}
\def\R{\mathbb{R}}
\def\orde{\omega}
\def\xor{\text{x}}

\renewcommand{\Sigma}{\Gamma}

\begin{document}
\title{On extremal factors of de Bruijn-like graphs}

\date{August 30, 2023}

\author{
    \small Nicol\'as \'Alvarez
    \and
    \small
    Ver\'onica Becher
    \and
    \small Mart\'in Mereb
    \and
    \small Ivo Pajor
    \and
    \small Carlos Miguel Soto
}

\maketitle

\begin{abstract}
    In 1972 Mykkeltveit proved that the maximum number of vertex-disjoint cycles in the de Bruijn graphs of order $n$ is attained by the pure cycling register rule, as conjectured by Golomb. We generalize this result to the tensor product of the de Bruijn graph of order $n$ and a simple cycle of size $k$, when $n$ divides $k$ or vice versa.
    We also develop  counting formulae for a large family of cycling register rules, including the linear register rules proposed by Golomb.
\end{abstract}

% 05C35  =>  Extremal problems in graph theory [See also 90C35]
% 05C38  =>  Paths and cycles [See also 90B10]
% 05C45  =>  Eulerian and Hamiltonian graphs


\section{Introduction and statement of results}
In~\cite[Chapter VII Conjecture A]{Golomb}  Golomb asked what is the  maximal number of vertex-disjoint cycles in the de Bruijn graph of order $n$. He conjectured that this is exactly the number of  cycles  attained using the pure cycling register  rule to partition the $n$-th de Bruijn graph.
This conjecture was proved by Mykkeltveit~\cite{Mukkelveit}, using Lempel's~\cite{Lempel} reformulation of the problem, which amounts to determining the minimum number of vertices which, if removed from the graph, will leave it with no cycles.

In this note we consider  Golomb's conjecture for a variant of the de Bruijn graph known as the {\em astute graph} for a fixed alphabet $\Sigma$.
These graphs are the tensor product of the de Bruijn graph of order $n$ and a simple cycle.
The Hamiltonian  cycles in the $(n,k)$-astute graph correspond to the so-called $(n,k)$-perfect necklaces introduced in~\cite{ABFY}.

The elements of $\Sigma^n$ are referred to as \emph{string}s of length $n$. The symbols in a string of length $n$ are numbered from $0$ to $n-1$.
The notation $s[i..j)$ for a string $s=a_0 a_1 \ldots a_{n-1}$ denotes the substring $a_i a_{i+1} \dots a_{j-1}$.

\begin{definition}[Astute graph]
    Given $n$ and $k$ positive integers, the astute graph is defined by $\Gnk = (\Vnk, \Enk)$, where $\Vnk = \Sigma^n \times \mathbb{Z}/k\mathbb{Z}$ and $\Enk$ is the set of all pairs
    \[((s, i), (t, j))\]
    \nopagebreak
    such that
    $s[1..n) = t[0..n-1)$
    and $j = i+1$.
\end{definition}

\begin{remark}
    Notice that $\Gast{n}{1}$ is the de Bruijn graph of order $n$. In this case, we identify the vertices $\Vast{n}{1}$ with $\Sigma^n$.
\end{remark}

\begin{definition}[Factor]
    A factor of $\Gnk$ is a set of vertex-disjoint cycles (directed circuits) which, together, include all the vertices of $\Gnk$. A factor which contains the maximum possible number of cycles is referred to as an {\em extremal factor}.
\end{definition}

To construct factors of de Bruijn and astute graphs,
we consider \emph{succession rules}.
These are what Golomb  in~\cite{Golomb} calls Shift Registers.

\begin{definition}[Succession rule]
    A succession rule is a bijective function $\sigma : \Sigma^n \to \Sigma^n$ such that for each string
    $s = a_0a_1\cdots a_{n-1}$,
    $\sigma(s) = a_1a_2\cdots a_{n-1} a_n$ for some $a_n \in \Sigma$.
\end{definition}

\begin{remark}
    The definition of succession rule implies that in the de Bruijn graph there exists an arc from $s$ to $\sigma(s)$, and in the astute graph there exists an arc from vertex $(s, i)$ to $(\sigma(s), i+1)$. This means that the succession rule $\sigma$ can be thought to act on the vertices of the de Bruijn and  astute graphs.
\end{remark}

\begin{definition}[Action of a succession rule on  astute graphs]
    Given a succession rule $\sigma$ and a positive integer $k$,  we  define an action $\Act{\sigma}{k} : \Vnk  \rightarrow \Vnk$ such that $\Act{\sigma}{k}(s,i) = (\sigma(s), i+1)$.
\end{definition}

Given a succession rule $\sigma$ and a positive integer $k$, the subgroup of permutations $\gen{\Act{\sigma}{k}}$ acts on $V_{n, k}$. For any vertex $v \in \Vast{n}{k}$, the arc $(v, \Act{\sigma}{k}(v))$ is in the graph $\Gast{n}{k}$. This implies that the orbits of this action are simple cycles on the astute graph.

\begin{definition}[Factor generated by succession rule]
    We denote $\Fac{\sigma}{k}$ as the factor composed of all orbits produced by $\Act{\sigma}{k}$.
\end{definition}

We interpret the alphabet $\Sigma$ as the ring $\mathbb{Z}/
    b\mathbb{Z}$ where $b = |\Sigma|$, therefore we can do linear arithmetic on its symbols.

\begin{definition}[Affine relation]
    A relation $R \subseteq \Sigma^{n+1}$ is said to be \emph{affine} if there exist  $c \in \Sigma$ and coefficients $(\lambda_i)_{0 \leq i \leq n} $, $\lambda_i\in \Sigma$, such that
    \[a_0 a_1 \cdots a_n \in R \iff c = \sum_{0 \leq i \leq n} \lambda_i a_i.\]
\end{definition}

\begin{definition}[Affine succession rule]
    An affine succession rule is a succession rule $\sigma : \Sigma^n \to \Sigma^n$ constructed from an affine relation $R$ as follows.
    For each string $a_0a_1\ldots a_{n-1}$,
    $\sigma(a_0a_1\ldots a_{n-1})$ is the unique string $a_1\ldots a_{n}$ such that $a_0a_1\ldots a_{n}$ is in $R$.
\end{definition}

\begin{remark}
    For an affine relation $R$ to give rise to an affine succession rule
    \begin{itemize}
        \item Each string must have at most one successor, which only happens if $\lambda_n$ is invertible; and
        \item The rule has to be bijective, which only happens if $\lambda_0$ is invertible.
    \end{itemize}
\end{remark}

\begin{example}[Pure Cycling Register]
    An example of an affine succession rule is the one given by string rotation. For any string $s = a_0a_1\cdots a_{n-1}$ we define
    \[
        \rot_n(s) = a_1\cdots a_{n-1}a_0,
    \]
    so, $a_n=a_0$.
    For $k=1$, the de Bruijn case, $\Fac{\rot_n}{1}$ is the set of necklaces of length $n$. Namely, the equivalence classes of $\Sigma^n$ under string rotation.
\end{example}

\begin{example}[Incremented Cycling Register]
    Another example of an affine succession rule is \emph{incremented rotation}. For any string $s = a_0a_1\cdots a_{n-1}$ we define
    \[
        \iota_n(s) = a_1\cdots a_{n-1}(a_0+1),
    \]
    so, $a_n=a_0+1$.
    An advantage of this particular succession rule $\iota$ is that each orbit has an equal quantity of each symbol in $\Sigma$. This has applications in the construction of de Bruijn sequences  with  small  discrepancy,  see~\cite{Huang} in contrast to~\cite{CooperHeitsch}.
\end{example}

\begin{example}[Xor Cycling Register]
    The third example we consider is restricted to the special case $|\Sigma| = 2$, where the ring addition operation is the \emph{xor}. In this ring, we define the succession rule for $s = a_0a_1\cdots a_{n-1}$ as
    \[
        \xor_n(s) = a_1\cdots a_{n-1}a_n
    \]
    where $a_n = a_1 + a_2 + a_3 + \cdots + a_{n-1}$.
\end{example}

We now state the main result  of this note:
\begin{theorem}\label{thm:extremal}
    Let $n$ and $k$ be positive integers such that $k$ divides $n$ or $n$ divides $k$. The factor $\Fac{\rot_n}{k}$ produced by the pure cycling register rule $r_n$ is extremal.
\end{theorem}

When $k$ does not divide $n$ Theorem~\ref{thm:extremal} is not necessarily true. For the case $k=2$, $n=3$ and $\Sigma = \set{0, 1}$, the successor rule $r_3$ produces a factor of size $4$ as shown in Figure~\ref{fig:counterexample_pcr}, while the extremal factors have size $6$. An example of such extremal factor is shown in Figure~\ref{fig:counterexample_extremal}. For such cases, where the hypothesis of Theorem~\ref{thm:extremal} do not hold, it remains an open problem to characterize the values of $k$ and $n$ where the conclusion does hold.
\medskip

The second result in this note is a closed formula for  the size of the factors generated by affine rules.

We use $(a : b)$ for  the greatest common divisor of the integers $a$ and $b$.
We write     $(P,Q)$ for the  ideal generated by $P $ and $Q$.
\begin{theorem}\label{thm:count}
    Let $n$ and $k$ be positive integers, and $R$ be an affine rule given by
    \[a_0 a_1 \cdots a_n \in R \iff c = \sum_{0 \leq i \leq n} \lambda_i a_i,\]
    for some coefficients $(\lambda_i)_i$ and a constant term $c$.
    Then the number of factors in the associated succession rule $\sigma$ is given by
    \begin{equation}\nonumber
        |\Fac{\sigma}{k}| = \frac{k\GCD{s}{ \orde}}{s\orde } \sum_{\GCD{s}{\orde} | d | \orde} \varphi \paren{\orde / d}\ideal{d}
    \end{equation}
    where
    \begin{itemize}
        \item  $\Lambda = \sum_{i=0}^n \lambda_i X^{n-i}$ is the characteristic polynomial of $R$,

        \item  $\orde$ is any multiple of the order of $X$ modulo $\Lambda$,

        \item  $\varphi$ is Euler's totient function, and

              $s$ is the length of the smallest cycle in the factor. Equivalently, $s$ can be defined as the smallest multiple of $k$ such that
              \[
                  c(1 + X + \cdots + X^{s-1}) \in (\Lambda, X^s-1).
              \]
    \end{itemize}
\end{theorem}

\begin{center}
    \begin{figure}[h!]
        \begin{subfigure}{0.45\textwidth}
            \includegraphics[width=0.99\linewidth]{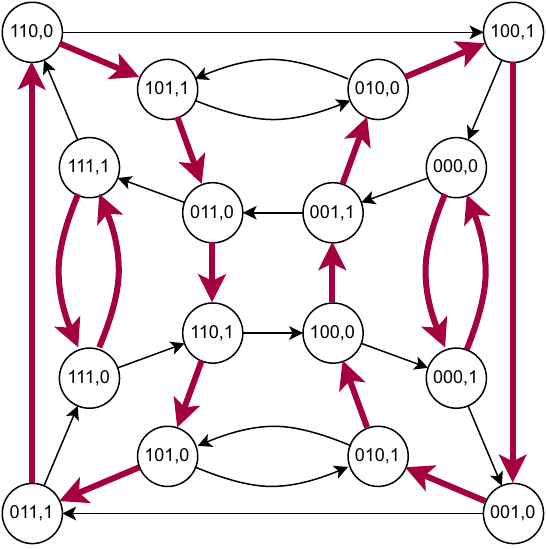}
            \caption{Factor $F_2(\rot_3)$ produced by Pure Cycling Register rule for $\Sigma = \set{0, 1}$. The arcs of the cycles in the factor are shown in magenta. There are 4 cycles in this factor.}
            \label{fig:counterexample_pcr}
        \end{subfigure}
        \quad\quad\quad
        \begin{subfigure}{0.45\textwidth}
            \includegraphics[width=0.99\linewidth]{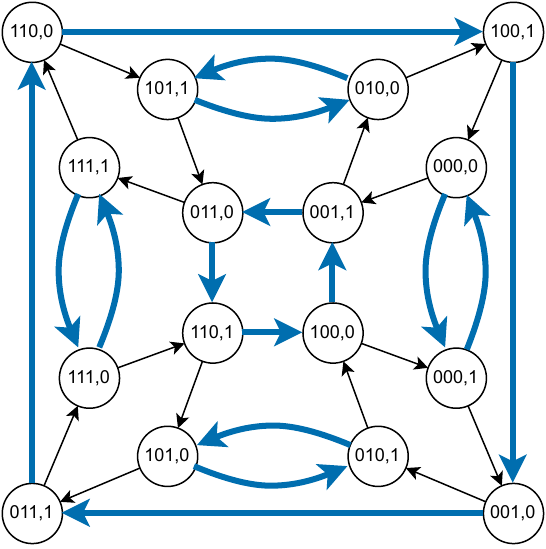}
            \caption{Extremal factor of $\Gast{3}{2}$ for the alphabet $\Sigma = \set{0, 1}$. The arcs of the cycles in the factor are shown in blue. There are 6 cycles in this factor, making it extremal.}
            \label{fig:counterexample_extremal}
        \end{subfigure}
        \caption{Pure Cycling Register induced factors may not be  extremal in astute graphs}
        \label{fig:image2}
    \end{figure}
\end{center}

The next corollaries give the number of elements in the factors determined by two specific succession rules.

\begin{corollary}[Factor of $\Gnk$ from Pure Cycling Register]
    \label{cor:coro1}
    \[
        |\Fac{\rot_n}{k}| =\frac{\GCD{n}{ k}}{n} ~\cdot~\sum_{\GCD{n}{ k}~\mid d~\mid n}~\varphi(n/d) |\Sigma|^d.
    \]
\end{corollary}
\begin{remark}
    When $k=1$ and $|\Sigma|=2$, $|\Fac{\rot_n}{1}| $ is the
    number of binary irreducible polynomials whose degree divides $n$,
    see \cite{a31}.
\end{remark}

\begin{corollary}[Factor of $\Gnk$ from Incremented Cycling Register]
    \label{thm:coro2}
    \begin{equation}\nonumber
        |\Fac{\iota_n}{k}| = \frac{k\GCD{\lcm(k, b d_b(n))}{n}}{\lcm(k, b d_b(n))n} \sum_{\GCD{\lcm(k, b d_b(n))}{n} | d | n} \varphi (n/d) b^d
    \end{equation}
    where $b=|\Sigma|$ and $d_b(n)$ is the smallest divisor of $n$ such that $n/d_b(n)$ is coprime with $b$.

\end{corollary}
\begin{remark}
    When $k=1$ and $|\Sigma|=2$, $|\Fac{\iota_n}{1}| $
    is the  number of distinct  output sequences
    from binary $n$-stage shift register which feeds back the complement of the last stage,
    see \cite{a16}.
\end{remark}

\begin{corollary}[Factor of $\Gnk$ from Xor Cycling Register]
    \label{cor:coro3}
    \begin{equation}\nonumber
        |\Fac{\xor_n}{k}| = \frac{k}{2(n+1)} \sum_{d | n+1} \varphi(2d) 2^{(n+1)/d}.
    \end{equation}
\end{corollary}
\begin{remark}
    When $k=1$ and $|\Sigma|=2$, $|\Fac{\xor_n}{1}| $ is the
    number of output sequences from $(n-1)$-stage shift register which feeds back the mod $2$ sum of the contents of the register, see \cite{a13}.
\end{remark}

\section{Proof of Theorem~\ref{thm:extremal}}
When $n$ divides $k$, the result follows from the fact that the pure cycle register produces a factor where each cycle has length exactly $k$, which is also the smallest possible length of a cycle in the graph $\Gnk$. Let us then consider the case $k$ divides $n$.

We use a basic tool from finite Fourier analysis  \cite{Beck}.

\begin{definition}[Discrete Fourier Transform]
    Let $\mu = e^{{2\pi}/{n}}$ be a primitive root of unity of order~$n$.
    Let us define $C: \Sigma^n \to \mathbb{C}$ as
    \[
        C(a_0\ldots a_{n-1}) = \sum_{i = 0}^{n-1} a_i\mu^i.
    \]
\end{definition}
\begin{lemma}
    \label{thm:caballo}
    $C(\rot_n(s)) = \mu^{-1} C(s)$.
\end{lemma}
\begin{proof}
    Let $s = a_0a_1\cdots a_{n-1}$. Then $\rot(s) = a_1a_2\cdots a_n a_0$. Then, we have:
    \[
        C(\rot(s)) = a_0\mu^{n-1} + \sum_{i = 1}^{n-1} a_i \mu^{i-1}.
    \]
    Since $\mu^{n-1} = \mu^{-1}$, we obtain
    \[
        C(\rot(s)) = \sum_{i = 0}^{n-1} a_i \mu^{i-1} = \mu^{-1} \sum_{i = 0}^{n-1} a_i \mu^i = \mu^{-1} C(s).
    \]
\end{proof}

\begin{lemma}
    \label{thm:pez}
    Let $(s_0, m_0), ... (s_{t-1}, m_{t-1})$
    be the vertices of any cycle in the astute graph
    $\Gnk$. Then $\sum C(s_i) = 0$.
\end{lemma}
\begin{proof}
    We have that
    \[\sum_{i=0}^{t-1} C(s_i) = \sum_{i=0}^{t-1} \sum_{j=0}^{n-1} (s_i)_j \mu^j.
    \]
    Since the strings $s_i$ form a cycle in the de Bruijn graph, we have that $(s_i)_j = (s_{i+1})_{j-1}$, where the indices of $s$ are taken modulo $t$. Then, $(s_i)_j$ depends only of the sum $i+j$ modulo~$t$.
    Let $w_{i+j} = (s_i)_j$ with the indices of $w$ taken modulo $t$. We can rewrite the expression as
    \[\sum_{i=0}^{t-1} C(s_i) = \sum_{i=0}^{t-1} \sum_{j=0}^{n-1} w_{i+j} \mu^j.
    \]
    With a change of variables $d = i+j$, we get
    \[\sum_{i=0}^{t-1} C(s_i) = \sum_{d=0}^{t-1} \sum_{j=0}^{n-1} w_d \mu^j
        = \paren{\sum_{d=0}^{t-1} w_d} \paren{\sum_{j=0}^{n-1} \mu^j}.
    \]
    And the sum of the powers of a primitive root is $0$, so we are done.
\end{proof}
\begin{lemma}
    \label{thm:lagarto}
    Let $s$ and $t$ be two strings that are connected by an arc in the de Bruijn graph~$\Gast{n}{1}$. Then $C(s) - C(\rot_n^{-1}(t)) \in \R$, and it is zero exactly when $s = \rot_n^{-1}(t)$.
\end{lemma}
\begin{proof}
    Since $s$ and $t$ are connected in the de Buijn graph, we can write
    \[s = a_0a_1\cdots a_{n-1}\]
    \[t = a_1a_2\cdots a_{n}.\]
    Then,
    \[\rot_n^{-1}(t) = a_n a_1\cdots a_{n-1}.
    \]
    Expanding the definition of $C(r^{-1}(t))$ we get
    \[C(\rot_n^{-1}(t)) = a_n \mu^0 + a_1 \mu^1 + a_2 \mu^2 + \cdots + a_{n-1} \mu^{n-1}.\]
    And for $C(s)$ we get
    \[C(s) = a_0 \mu^0 + a_1 \mu^1 + a_2 \mu^2 + \cdots + a_{n-1} \mu^{n-1}.
    \]
    So, we have $C(s) - C(\rot_n^{-1}(t)) = a_0 - a_n$, which is a real number that is zero only when $a_0 = a_n$, which is precisely when $s = \rot_n^{-1}(t)$.
\end{proof}

To prove that the factor produced by the Pure Cycling Register is extremal, we  choose for each cycle in the factor $\Fac{\rot_n}{k}$ a distinguished vertex, and then we prove that any cycle in any factor has at least one distinguished vertex,
therefore the size of any factor is at most the number of distinguished vertices.

Let $(s_0, m_0), (s_1, m_1), \cdots, (s_{t-1}, m_{t-1})$ be any cycle in $\Fac{\rot_n}{k}$.     There are two possibilities.

One possibility is that
the transform $C(s_i)$ is real for all $s_i$. % in the factor. 
In this case we  take any arbitrary vertex in the factor as the distinguished vertex.

The other possibility is that there exists some string $s_i$ such that $C(s_i)$ is not real.
Let $z = C(s_0)$. Due to Lemma~\ref{thm:caballo}, we have that $C(s_i) = z\mu^{-i}$. Since the length of any PCR cycle is a divisor of $\lcm(n, k) = n$
and $z\mu^{-i}$ has a cycle length equal to the order of $\mu$ (which is $n$), the size of the factor must be $t=n$, and the transforms of its strings form a regular $n$-sided polygon on the complex plane.

\begin{figure}[h]
    \begin{center}
        \begin{subfigure}{0.5\textwidth}
            \includegraphics[width=0.99\linewidth]{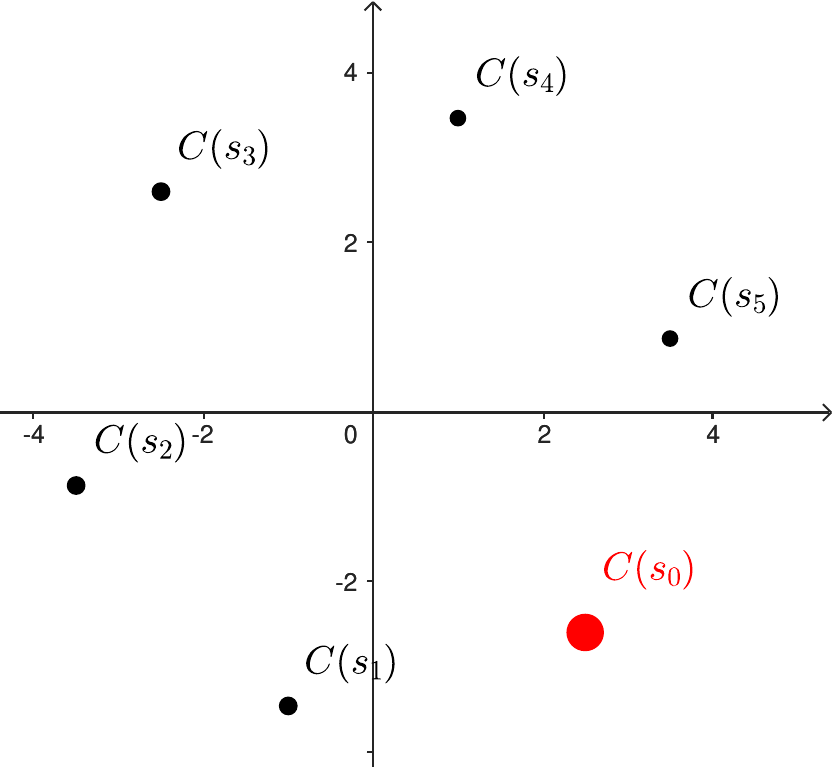}
        \end{subfigure}
    \end{center}
    \caption{Transforms of the strings on the PCR cycle generated by $s_0=123351$. The large red point is the distinguished vertex for this PCR cycle.}
    \label{fig:rata}
\end{figure}

The distinguished vertex of the factor will be the unique vertex $(s_i, m_i)$ such that \[\Im {C(s_i)} < 0 \quad \text{but} \quad \Im {C(s_{i-1})} \geq 0,\] as exemplified in Figure~\ref{fig:rata}.
Now we have to prove that every cycle in the astute graph $\Gast{n}{k}$ contains at least one distinguished vertex. Let \[(s_0, m_0), (s_1, m_1), \cdots, (s_{t-1}, m_{t-1})\]
be any such cycle.
We consider three cases:

\paragraph{First case: There exists some string $s_i$ such that $C(s_i)$ is not real.}
The sum of $C(s_i)$ over all $i$ must be zero due to Lemma~\ref{thm:pez}. So, if the transforms are not always real, there must be a string where the imaginary part of the transform is positive and another where the imaginary part is negative. In particular, let $s_i$ be any string such that $\Im {C(s_i)} < 0$ but $\Im {C(s_{i-1})} \geq 0$.
Since $s_{i-1}$ and $s_i$ are connected by an arc, Lemma~\ref{thm:lagarto} implies that $C(s_{i-1})$ and $C(\rot_n^{-1}(s_i))$ differ by a real number. And since $\Im {C(s_{i-1})} \geq 0$, then $\Im{C(\rot_n^{-1}(s_i))} \geq 0$ as well. This means that in the PCR cycle
\[(s_i, m_i), (\rot_n(s_i), m_i+1), (\rot_n^2(s_i), m_i+2), \dots, (\rot_n^{n-1}(s_i), m_i + n - 1) = (\rot_n^{-1}(s_i), m_i - 1)\]
the distinguished vertex will be $(s_i, m_i)$, which is a vertex in the original cycle
\[(s_0, m_0), (s_1, m_1), \cdots, (s_{t-1}, m_{t-1}).\]

\paragraph{Second case: The transform $C(s_i)$ is $0$ for all $i$.}
Due to Lemma~\ref{thm:lagarto}, $s_i = \rot_n^{-1}(s_{i+1})$ if and only if $C(s_i) = C(\rot_n^{-1}(s_{i+1}))$. By Lemma~\ref{thm:caballo}, $C(\rot_n^{-1}(s_{i+1})) = \mu C(s_{i+1}) = \mu \cdot 0 = 0$
, so $s_i = \rot_n^{-1}(s_{i+1})$ for all $i$. Hence,  the cycle is actually a PCR cycle, so it must have at least one distinguished vertex.

\paragraph{Third case: The transform $C(s_i)$ is always real, but not always $0$.}
Let $C(s_i) \in \R - \{0\}$. Due to Lemma~\ref{thm:lagarto}, $C(\rot_n^{-1}(s_i)) = \mu C(s_i)$ differs from $C(s_{i-1})$ by a real number. If $n > 2$, $\mu$ is a non-real complex number, and therefore $\mu C(s_i)$ also has a non-zero imaginary component, which is equal to that of $C(s_{i-1})$. This is a contradiction, because we assumed $C(s_j)$ was real for all $j$.
So we only have to analyze the cases $n=1$ and $n=2$. In both cases Theorem~\ref{thm:extremal} is implied by the fact that the PCR cycle of a vertex is the smallest possible cycle it belongs to in the astute graph.

We showed that every cycle in  graph $G_{n,k}$
contains at least one distinguished vertex of a cycle in the factor determined by the Pure Cycling Register rule, Theorem \ref{thm:caballo}
is proved.

\section{Proof of Theorem~\ref{thm:count}}

\subsection{Burnside's Lemma}
Our main tool for counting the number of cycles in a succession rule is the classical Burnside's Lemma~\cite{burnside}.
It states that for any finite group~$G$ acting on a set~$S$, the following identity holds:
\[|S/G| = \frac{1}{|G|}~\sum_{g\in G} |S^g|,
\]
where $S/G$ is the set of orbits of $S$ under the action of $G$, and $S^g$ is the subset of $S$ fixed by the action of $g$.

When considering succession rules, the set $S$ is the set of vertices of an astute graph $S=\Vnk$ and $G = \gen{\Act{\sigma}{k}}$ is the group generated by the action associated with a succession rule $\sigma$. In that case, the set of orbits $S/G$ coincides with the factor $\Fac{\sigma}{k}$. We thus have the following identity:

\begin{equation} %\nonumber
    \label{eqn:patorde}
    |\Fac{\sigma}{k}| = \frac{1}{\orde } \sum_{i=0}^{\orde -1} |\Fix(\Act{\sigma}{k}^i)|,
\end{equation}

where $\orde$ is the order of $\Act{\sigma}{k}$ and $\Fix(f)$ is the set of fixed points of the function $f$.

Notice that the function $i \mapsto |\Fix(\Act{\sigma}{k}^i)|$ is defined over all the integers, and it is cyclic because $\Act{\sigma}{k}^{i} = \Act{\sigma}{k}^{i+\orde }$ for all $i\in\mathbb{Z}$. Therefore,  Equation~\eqref{eqn:patorde} asserts that the size of $\Fac{\sigma}{k}$ is the average of the function $i \mapsto |\Fix(\Act{\sigma}{k}^i)|$ over one cycle.
This average does not depend on which cycle is picked, because they all coincide with the average of the function $i \mapsto |\Fix(\Act{\sigma}{k}^i)|$ in the range $[0, t]$ when $t$ tends to infinity. This gives rise to the following lemma.

\begin{lemma}
    \label{thm:osorde}
    Let $k$ be a positive integer and $\sigma : \Sigma^n \to \Sigma^n$ a succession rule. Let $\orde$ be any positive integer such that
    \[|\Fix(\Act{\sigma}{k}^{i})| = |\Fix(\Act{\sigma}{k}^{i+\orde })|, \quad \forall i \in \mathbb{Z}.\]
    Then,
    \[|\Fac{\sigma}{k}| = \frac{1}{\orde } \sum_{i=0}^{\orde -1} |\Fix(\Act{\sigma}{k}^i)|.\]
\end{lemma}

For the de Bruijn case, since we identify $\Vast{n}{1}$ with $\Sigma^n$, we have that $\Act{\sigma}{1} = \sigma$, and therefore $|\Fix(\Act{\sigma}{1}^i)| = |\Fix(\sigma^i)|$. Let us analyze $|\Fix(\Act{\sigma}{k}^i)|$ in the astute case.

Let $(s, j) \in \Vnk$ be any vertex fixed by $\Act{\sigma}{k}^i$. We have that
\[(s, j) = \Act{\sigma}{k}^i(s, j) = (\sigma^i(s), j+i).
\]
So $(s, j)$ is a fixed point of $\Act{\sigma}{k}^i$ if and only if  $s$ is a fixed point of $\sigma^i$ and $k | i$, which gives us the following identity:
\[|
    \Fix(\Act{\sigma}{k}^i)| = k \One{k|i} |\Fix(\sigma^i)|,
\]
where $\One{p}$
is $1$ when $p$ is true, and $0$ otherwise.
Rewriting Lemma~\ref{thm:osorde} with this identity we obtain the following.

\begin{lemma}
    \label{thm:panda}
    Let $k$ be a positive integer and $\sigma : \Sigma^n \to \Sigma^n$ a succession rule. Let $\orde$ be any positive integer such that
    \[k \One{k|i}|\Fix(\sigma^i)| = k \One{k|i+\orde } |\Fix(\sigma^{i+\orde })|,  \quad \forall i \in \mathbb{Z}.\]
    Then,
    \[|\Fac{\sigma}{k}| = \frac{k}{\orde } \sum_{i=0}^{\orde -1} \One{k|i} |\Fix(\sigma^i)|.\]
\end{lemma}

\subsection{GCDs of certain families of polynomials}

Here we deal with the families of polynomials $X^n-1$ and $U_n = \frac{X^n-1}{X-1}$. Although in
an arbitrary ring not every pair of polynomials has a GCD,
we  show that any pair of polynomials in these classes has a GCD, and we compute it explicitly.
\begin{lemma}
    \label{lemma:unum}
    $(U_n : U_m) = U_{(n : m)}$.
\end{lemma}
\begin{proof}
    Assuming that $n < m$, we have the following identity:
    \[U_m - X^{m-n}U_n = U_{m-n}.
    \]
    Then, $(U_n : U_m) = (U_n : U_{m-n})$.
    It is also true that $(n:m) = (n, m-n)$. This means that applying the euclidean algorithm over the polynomials $U_n$ and $U_m$ mirrors the steps taken during the application of the algorithm to the pair of integers $n$ and $m$, which implies that the algorithm always terminates, and converges to $U_{(n:m)}$
\end{proof}

\begin{lemma}
    \label{lemma:xnxm}
    $(X^n - 1 : X^m - 1) = X^{(n:m)} - 1$.
\end{lemma}
\begin{proof}
    It follows directly from Lemma~\ref{lemma:unum} by multiplying both sides by $X-1$.
\end{proof}

\begin{lemma}
    \label{lemma:unxm}
    If $\Sigma$ is a field, then
    \[
        (U_n : X^m - 1) = \begin{cases}
            X^{(n:m)}-1 & \text{ if } n / (n:m) \equiv 0 \mod |\Sigma|      \\
            U_{(n:m)}   & \text{ if } n / (n:m) \not\equiv 0 \mod |\Sigma|. \\
        \end{cases}
    \]
\end{lemma}
\begin{proof}

    Since $\Gamma[X]$ is a principal ideal and
    $
        (X^n-1, X^m-1) \subseteq
        (U_n, X^m-1) \subseteq
        (U_n, U_m)
    $
    we only need to decide whether $(U_n, X^m-1)$ is generated by $X^{(n:m)}-1$ or by $U_{(n:m)}$. The former is true if and only if $X^{(n:m)}-1 \mid U_n$.

    Since
    \[
        U_n
        = \frac{X^{(n:m)} - 1}{X-1}
        U_{n/(n:m)}(X^{(n:m)})
    \]
    we get  $X^{(n:m)}-1 \mid U_n$ if and only if
    $1$ is a root of $U_{n/(n:m)}(X^{(n:m)}).$ The latter is equivalent to $n/(n:m)=U_{n/(n:m)}(1)=0$ in $\Gamma.$

\end{proof}

\subsection{Affine Succession Rules}

%        The goal of this section is 
In this section we  compute the size of the set $\Fix(\sigma^k) \inn \Gamma^n$ for an affine succession rule $\sigma : \Gamma^n \to \Gamma^n$ given by an affine relation $R$.
For any $s \in \Fix(\sigma^k)$ we define its associated string $w \in \Gamma^k$ by
\[w_i = (\sigma^i(s))_0\]
It is clear that $w$ uniquely determines $s$. Indeed, due to the definition of succession rule,
\[
    s_i = (\sigma^i(s))_0 = w_{(i \mod k)},\]
%the latter equality being true only because 
this identity holds only because
$s \in \Fix(\sigma^k)$.
% and therefore the 
Hence, the strings $\sigma^i(s)$ repeat modulo $k$.
Equivalently, if we write $w^* = w w w w w \cdots$ as the infinite concatenation of $w$ with itself, the above claim states that $s = \Substr{w^*}{0}{n}$. A similar argument shows that $\sigma^i(s) = \Substr{w^*}{i}{n+i}$. Since $\sigma$ is an affine succession rule we have  that every $n+1$ consecutive symbols in $w^*$ satisfy the affine relation $R$. This condition can be encoded
%codified 
in the following polynomial series equation:
\begin{equation}
    \label{eqn:tigre}
    \exists p :\deg(p) \leq n \text{ and }\quad \frac{w}{1 - X^k} \Lambda = p + \frac{c}{1 - X},
\end{equation}
where $\Lambda$ is the characteristic polynomial of $R$, $c$ is its constant term, and we identify the string $w$ with its generating polynomial. Under this identification, $\frac{w}{1 - X^k}$ is the generating function of the infinite string $w^*$.

The coefficient of degree $i$ in $\frac{w}{1 - X^k} \Lambda$ is the linear combination
\[
    \sum_{j=0}^n \lambda_j w^*_{i-n+j}.
\]

%        That is the reason that 
Thus, when $i \leq n$ we cannot guarantee that the result of this linear combination is $c$, and we need to introduce an ``error'' polynomial $p$ of degree at most $n$.

If we multiply both sides of Equation~\eqref{eqn:tigre} by $1-X^k$ we get the equivalent expression:
\[\exists p :\deg(p) \leq n \text{ and }\quad w \Lambda = p(1 - X^k) + c \Polu_k\]
which is, by definition, the same as
\begin{equation}
    \label{eqn:alga}
    w \Lambda \equiv c \Polu_k \mod 1 - X^k.
\end{equation}
We claim that \emph{any} string $w \in \Gamma^k$ that satisfies Equation~\eqref{eqn:alga} is the associated string of some other string $s \in \Fix(\sigma^k)$. Indeed, take any $w$ that satisfies~\eqref{eqn:alga}. Since~\eqref{eqn:alga} is equivalent to~\eqref{eqn:tigre}, every substring of length $n+1$ of $w^*$ will satisfy the relation $R$. Then, if we take $s = \Substr{w^*}{0}{n}$ then $\sigma^i(s) = \Substr{w^*}{i}{n+i}$, which will be cyclic modulo $k$, because $w^*$ is cyclic modulo $k$.

The established bijection implies that $|\Fix(\sigma^k)|$ is the number of solutions to the Equation~\eqref{eqn:alga}. Since $w$ is always of length $k$, we can assume $w$ is a polyomial in the quotient ring $\Gamma[X]/(1 - X^k)$ and the number of solutions of the equation will not change.

For an arbitrary polynomial $P \in \Sigma[X] / (X^k-1)$, the system of equations
\[\Lambda w \equiv P \quad \mod X^k - 1\]
has a solution if and only if
\[u \cdot \Lambda - P = v \cdot (X^k-1)\]
for some polynomials $u, v$. This is equivalent to the condition
\[P \in (\Lambda, X^k-1).
\]
Furthermore, if the system does have a solution, it has the same number of solutions as the associated linear system
\[
    \Lambda w \equiv 0 \quad \mod X^k - 1.
\]
Since there are $|\Sigma[X] / (X^k - 1)|$ possibilities for $w$, and each possibility for $P$ gets an equal number of solutions, each $P$ gets exactly
\[
    \frac{|\Sigma[X] / (X^k - 1)|}{|(\Lambda, X^k-1) / (X^k - 1)|}
    = |\Sigma[X] / (\Lambda, X^k - 1)|
\]
solutions to the linear system, if there is at least one.
This implies that
\begin{equation}
    \nonumber
    \label{eqn:abeja}
    |\Fix(\sigma^k)| = \ideal{k} \One{c\Polu_k \in (\Lambda, X^k - 1)}.
\end{equation}

Let us further analyze the condition $c\Polu_k \in (\Lambda, X^k - 1)$. Let
\[
    S = \{k \in \mathbb{N} : c\Polu_k \in (\Lambda, X^k - 1)\}.
\]
We prove that $S$ the set of all multiples of its least element ${\ell_\sigma}$.
Let $d$ be any multiple of ${\ell_\sigma}$. Since ${\ell_\sigma} \in S$ we have that:
\[c\Polu_{{\ell_\sigma}} \in (\Lambda, X^{\ell_\sigma} - 1) = (\Lambda, (X-1)\Polu_{{\ell_\sigma}})\]
The condition ${\ell_\sigma} | d$ implies $\Polu_{\ell_\sigma} | \Polu_d$ and therefore it also holds that
\[c\Polu_d \in \frac{\Polu_d}{\Polu_{\ell_\sigma}}(\Lambda, (X-1)\Polu_{\ell_\sigma}) \inn \left(\Lambda, \frac{\Polu_d}{\Polu_{\ell_\sigma}}(X-1)\Polu_{\ell_\sigma}\right) = (\Lambda, (X-1)\Polu_d)= (\Lambda, X^d - 1).\]
Hence, $d \in S$.

Now take any $d \in S$, and let $g = \gcd(d, {\ell_\sigma})$. Then $U_g \in (U_{\ell_\sigma}, U_d)$ due to theorem \ref{lemma:unum}. We also have that
\[cU_{\ell_\sigma} \in (\Lambda, X^{\ell_\sigma} - 1) \quad \text{and} \quad cU_d \in (\Lambda, X^d - 1).\]
Therefore,
\[c\Polu_g \in (\Lambda, X^{\ell_\sigma} - 1, X^d - 1) = (\Lambda, X^g - 1).\]
Thus, $g \in S$. Since $g | {\ell_\sigma}$ and ${\ell_\sigma}$ is the least element in $S$, then $g = {\ell_\sigma}$ and, as a result, $d$ is a multiple of ${\ell_\sigma}$.

Notice that $\ell_\sigma$ is the smallest-length cycle in the associated succession rule $\sigma$, since it is the first integer for which $\Fix(\sigma^{\ell_\sigma})$ is nonempty. Observe that when $c = 0$, the zero string always has cycle length one, so $\ell_\sigma = 1$.

We  rewrite Equation~\eqref{eqn:abeja} as
\[
    \Fix(\sigma^k) = \One{\ell_\sigma | k} \cdot \ideal{k}.
\]
Now let us analyze $\ideal{k}$. Since the first coefficient of $\Lambda$ is invertible, the polynomial $X$ is invertible modulo $\Lambda$, so there exists some $\orde$ such that $X^\orde \equiv 1 \mod \Lambda$. Equivalently, $\Lambda | X^\orde - 1$.

Let $k$ be any positive integer and $g = \GCD{\orde}{ k}$. We know that since ${\GCD{k}{ \orde}} | k$, $X^{\GCD{k}{ \orde}} - 1 | X^k - 1$. Therefore,
\[(\Lambda, X^k - 1) \subseteq (\Lambda, X^{\GCD{k}{ \orde}} - 1).
\]
And also,
\[(X^{\GCD{k}{ \orde}}-1) = (X^k-1, X^\orde - 1) \inn (\Lambda, X^k-1).\]
Consequently,  the ideals $(\Lambda, X^k - 1)$ and $(\Lambda, X^{\GCD{k}{ \orde}} - 1)$ coincide.
When we replace this into the formula for $\Fix(\sigma^k)$, we get the following.

\begin{lemma}
    \label{thm:cyc_ideal}
    \[
        \Fix(\sigma^k) = \One{\ell_\sigma | k}\cdot \ideal{\GCD{k}{ \orde}}.
    \]
\end{lemma}

\subsection{Burnside's Lemma for affine necklaces}

Let $k$ be a positive integer and let $\sigma$ be an affine succession rule with its characterisitc polynomial $\Lambda$. Due to Lemma~\ref{thm:panda}, we have that
\begin{equation}
    \label{eqn:pulpo}
    |\Fac{\sigma}{k}| = \frac{k}{M} \sum_{i=0}^{M-1} \One{k|i} |\Fix(\sigma^i)|,
\end{equation}
for any $M$ that is a cycle of the function $i \mapsto \One{k|i} |\Fix(\sigma^i)|$.
Due to Lemma~\ref{thm:cyc_ideal}, we know that if $\orde$ is the order of $X$ modulo $\Lambda$, then
\[\Fix(\sigma^k) = \One{\ell_\sigma | k} \cdot \ideal{\GCD{k}{ \orde}}.
\]
Therefore, $M$ needs to be a cycle of
\[i \mapsto \One{k|i} \cdot \One{\ell_\sigma | i} \cdot \ideal{\GCD{i}{ \orde}}.\]
To be a cycle of a product of functions, it suffices to be a multiple of the cycle length of each factor. So,  we have that one possible value of $M$ is $\lcm(k, \ell_\sigma, \orde)$.
Now rewriting Equation~\eqref{eqn:pulpo} we get

\begin{align*}
    |\Fac{\sigma}{k}| & = \frac{k}{M} \sum_{i=0}^{M-1} \One{k|i} \cdot \One{\ell_\sigma | k} \cdot \ideal{\GCD{i}{ \orde}}
    \\&=\frac{k}{M} \sum_{i=0}^{M-1} \One{\lcm(i, \ell_\sigma)|i} \cdot \ideal{\GCD{i}{ \orde}}.
\end{align*}

Recall that each summand in the Burnside equation corresponds to $\Fix(\Act{\sigma}{k}^i)$, and the size of the ideal vector space $\ideal{\GCD{i}{ \orde}}$ is always positive. Hence, $\Fix(\Act{\sigma}{k}^i)$ is zero if and only if $\One{\lcm(i, \ell_\sigma)|i}$ is zero. We conclude that $\lcm(i, \ell_\sigma)$ is the length of the smallest cycle in the factor $\Fac{\sigma}{k}$. Let $S$ be that cycle length. Rewriting the equation above  we obtain the following:
\begin{align*}
    |\Fac{\sigma}{k}| & = \frac{k}{M} \sum_{i=0}^{M-1} \One{S|i} \cdot \ideal{\GCD{i}{ \orde}}
    \\
                      & = \frac{k}{M} \sum_{i=0}^{M/S-1} \ideal{\GCD{iS}{ \orde}}
    \\
                      & =\frac{k}{M} \sum_{i=0}^{M/S-1} \ideal{\GCD{i}{\orde/\GCD{S}{ \orde}} \cdot \GCD{S}{ \orde}}
    \\
                      & =\frac{k}{S \orde /\GCD{S}{ \orde}} \sum_{i=0}^{\orde /\GCD{S}{ \orde}-1} \ideal{\GCD{i}{\orde/\GCD{S}{ \orde}} \cdot \GCD{S}{ \orde}}.
\end{align*}

The second equality uses that  $S|M$,
and the last equality uses that since
$M = \lcm(S, \orde)$, we have $\orde/\GCD{S}{ \orde} = M/S$.

Since $\gcd(i, \orde/\GCD{S}{ \orde})$ iterates over all divisors of $d$, we can express that sum as follows:
\begin{equation}\nonumber
    |\Fac{\sigma}{k}| = \frac{k\GCD{S}{ \orde}}{S\orde} \sum_{d | \frac{\orde }{\GCD{S}{ \orde}}} \varphi \paren{\frac{\orde }{d\GCD{S}{ \orde}}}\ideal{d \cdot \GCD{S}{ \orde}}.
\end{equation}
This can be rewritten as        \begin{equation}\nonumber
    |\Fac{\sigma}{k}| = \frac{k\GCD{S}{ \orde}}{S \orde} \sum_{\GCD{S}{\orde} | d | \orde} \varphi \paren{\orde / d}\ideal{d}.
\end{equation}
This completes the proof of Theorem~\ref{thm:count}.

\section{Proof of the Corollaries}
\subsection{Proof of Corollary~\ref{cor:coro1}}
\begin{proof}[Proof of Corollary~\ref{cor:coro1}]
    The PCR rule for necklaces of order $n$ is affine, and its associated affine relation is given by
    \[(a_i)_i \in R \iff 0 = a_0 - a_n.
    \]
    Then, its characteristic polynomial is $\Lambda = X^n - 1$.
    Using Theorem~\ref{thm:count},
    \begin{equation}\nonumber
        |\Fac{\rot_n}{k}| = \frac{k\GCD{s}{s\orde}}{\orde} \sum_{\GCD{s}{\orde} | d | \orde} \varphi (\orde/d) \ideal{d}
    \end{equation}
    where
    \begin{itemize}
        \item $\orde$ is the order of $X$ modulo $\Lambda$,

        \item $\varphi$ is Euler's totient function, and

        \item  $s$ is the length of the smallest cycle in the factor. Equivalently, $s$ can be defined as the least multiple of $k$ such that
              \[
                  c(1 + X + \cdots + X^{s-1}) \in (\Lambda, X^s-1).
              \]

              In this case the associated constant $c$ is $0$. That is, the rule is linear. Therefore $s = k$ because the condition $c(1 + X + \cdots + X^{s-1}) \in (\Lambda, X^s-1)$ always holds.
    \end{itemize}
    The order of $X$ modulo $X^n - 1$ is $n$; therefore, $g = \GCD{n}{k}$. Replacing these identities into the formula for $|\Fac{\rot_n}{k}|$ we get

    \begin{equation}\nonumber
        |\Fac{\rot_n}{k}| = \frac{k\GCD{k}{n}}{kn} \sum_{\GCD{k}{n} |d | n} \varphi(n/d) \ideal{d}.
    \end{equation}
    Observe that  $(\Lambda, X^{d}-1) = (X^{\GCD{n}{d}} - 1) = (X^d - 1)$.
    Hence,
    \[
        \left|\frac{\Sigma[X]}{(\Lambda, X^{d} - 1)}\right| =
        \left|\frac{\Sigma[X]}{(X^d - 1)}\right| =
        b^d.
    \]
    Replacing this into the formula for $|\Fac{\rot_n}{k}|$ we get
    \begin{equation}\nonumber
        |\Fac{\rot_n}{k}| = \frac{\GCD{k}{n}}{n} \sum_{\GCD{k}{n} |d | n} \varphi(n/d) b^d.
    \end{equation}
\end{proof}

\subsection{Proof of Corollary~\ref{thm:coro2}}
\begin{proof}[Proof of Corollary~\ref{thm:coro2}]
    For the incremented cycle register case, the associated characteristic polynomial is $\Lambda = X^n - 1$ as in the PCR case, but the constant $c$ is $1$ instead of $0$.
    To specialize Theorem~\ref{thm:count}, we have to find the smallest integer $s$ that is a multiple of $k$ and
    \[
        1 + X + \cdots + X^{s-1} \in (\Lambda, X^s-1).
    \]
    Let $d = \GCD{n}{s}$.
    Notice that $\Lambda = X^n - 1$ and so $(\Lambda, X^s-1) = (X^d - 1)$. Since the ideal is principal, the condition $1 + X + \cdots + X^{s-1} \in (X^d - 1)$ can be checked by reducing $1 + X + \cdots + X^{s-1}$ modulo the polynomial $X^d - 1$ and checking if the result is $0$.

    When we reduce a polynomial modulo $X^d - 1$, the $i$-th coefficient of the reduced polynomial is the sum of all the coefficients of the original polynomial that have degree congruent to $i$ modulo $d$.
    The polynomial $1 + X + \cdots + X^{s-1}$ has all coefficients equal to $1$. So, when reduced modulo $X^d - 1$, each resulting coefficient will be $s/d$, since for each $i \in \{0, \dots d-1\}$ there are $s/d$ indices in $\{0, \dots s\}$ congruent to $i$ modulo $d$.

    In order for $1 + X + \cdots + X^{s-1} \mod X^d - 1$ to be $0$, we need $s/d \equiv 0 \mod b$. So $s$ has to be a multiple of $b$, which we can write as $s = bm$ for some $m$. Furthermore, we need the following condition to hold
    \[s/d = bm / \GCD{bm}{n} \equiv 0 \mod b.\]
    This is equivalent to $\GCD{bm}{n} | m$, which in turn is equivalent to
    \[\GCD{b\frac{m}{\GCD{m}{n}}}{\frac{n}{\GCD{m}{n}}} \Big | \frac{m}{\GCD{m}{n}}.\]
    Notice that $\frac{m}{\GCD{m}{n}}$ is coprime with $\frac{n}{\GCD{m}{n}}$, so
    \[\GCD{b\frac{m}{\GCD{m}{n}}}{\frac{n}{\GCD{m}{n}}} = \GCD{b}{\frac{n}{\GCD{m}{n}}}.\]
    And the only divisor of $\frac{n}{\GCD{m}{n}}$ that is also a divisor of $\frac{m}{\GCD{m}{n}}$ is $1$, so the condition holds only when
    \[\GCD{b}{\frac{n}{\GCD{m}{n}}} = 1,\]
    which is true precisely when $d_b(n) | m$. Since we also require that $s$ be a multiple of $k$, the smallest possible value for $s$ is:
    \[s = \lcm(k, b d_b(n)).
    \]

    We apply Theorem \ref{thm:count} as we did in the proof of Corollary \ref{cor:coro1}, with $\omega=n$  and $        \left|\frac{\Sigma[X]}{(\Lambda, X^{d} - 1)}\right| =
        b^d.$
    For any divisor $d$ of the order $\omega = n$. Replacing this into the formula of Theorem~\ref{thm:count} we get
    \begin{equation}\nonumber
        |\Fac{\iota_n}{k}| = \frac{k\GCD{\lcm(k, b d_b(n))}{n}}{\lcm(k, b d_b(n))n} \sum_{\GCD{\lcm(k, b d_b(n))}{n} | d | n} \varphi (n/d) b^d.
    \end{equation}
\end{proof}

\subsection{Proof of Corollary~\ref{cor:coro3}}

\begin{proof}[Proof of Corollary \ref{cor:coro3}]
    The Xor rule for necklaces of order $n$ is affine, and its associated affine relation is given by
    \[(a_i)_i \in R \iff a_n = a_0 + a_1 + \cdots + a_{n-1}.
    \]
    Consequently, its characteristic polynomial is $\Lambda = 1 + X + \cdots + X^{n-1} - X^n$, which is equal to $\Polu_{n+1}$ when $|\Sigma| = 2$.
    To specialize Theorem~\ref{thm:count}, we need to compute:
    \begin{itemize}
        \item The length of the smallest cycle, which is $1$ since the rule is linear
        \item (A multiple of) the order $\orde$ of $X$ modulo $\Lambda$
        \item The size of $\ideal{d}$ for all divisors $d | \orde$.
    \end{itemize}

    We know that $\Lambda = U_{n+1}$ which divides $X^{n+1} - 1$. Therefore,  $\omega = n+1$ is a multiple of the order of $X$ modulo $\Lambda$.
    Due to Lemma~\ref{lemma:unxm},
    \[(\Lambda, X^d - 1) = ((U_\orde : X^d - 1)) = \begin{cases}
            X^d - 1 & \text{ if } \orde/d \text{ is even} \\
            U_d     & \text{ if } \orde/d \text{ is odd}. \\
        \end{cases}
    \]
    Hence,
    \[\ideal{d} = |\Sigma|^{d - 1 + \One{2 \;\mid\; \orde/d}}.\]
    Replacing this in the statement of Theorem~\ref{thm:count} we get
    \begin{align*}\nonumber
        |\Fac{\sigma}{k}| = & \; \frac{k\GCD{s}{ \orde}}{s\orde } \sum_{\GCD{s}{\orde} | d | \orde} \varphi \paren{\orde / d}\ideal{d}
        \\ =&\; \frac{k}{\orde} \sum_{d | \orde} \varphi \paren{\orde / d} |\Sigma|^{d - 1 + \One{2 \;\mid\; \omega/d}}.
        \\\intertext{If we do a change of variables $d \mapsto \omega/d$, and set $|\Gamma| = 2$ we get}
        =                   & \;
        \frac{k}{\orde} \sum_{d | \orde} \varphi(d) 2^{\One{2 \mid d}} 2^{\orde/d-1}.
        \\\intertext{Since $\varphi(2d) = d$ when $d$ is even and $\varphi(2d)=2d$ when $d$ is odd, $\varphi(d) 2^{\One{2 \mid d}}$ reduces to $\varphi(2d)$}
        =                   & \;
        \frac{k}{2\orde} \sum_{d | \orde} \varphi(2d) 2^{\orde/d}.
    \end{align*}

\end{proof}

\bibliographystyle{plain}

\bibliography{extremal.bib}

{\small
\noindent
Nicol\'as Álvarez \\
ICC CONICET Argentina -  {\tt  nico.alvarez@gmail.com}
\medskip

\noindent
Ver\'onica Becher \\
Departamento de  Computaci\'on, Facultad de Ciencias Exactas y Naturales \& ICC  \\
Universidad de Buenos Aires \&  CONICET  Argentina-  {\tt  vbecher@dc.uba.ar}
\medskip

\noindent
Martín Mereb \\
Departamento de Matemática, Facultad de Ciencias Exactas y Naturales \& IMAS \\
Universidad de Buenos Aires \&  CONICET Argentina-  {\tt  mmereb@gmail.com}
\medskip

\noindent
Ivo Pajor\\
Departamento de Computaci\'on, Facultad de Ciencias Exactas y Naturales \\
Universidad de Buenos Aires \&  Argentina-  {\tt  pajorivo@gmail.com}
\medskip

\noindent
Carlos Miguel Soto\\
Departamento de Computaci\'on, Facultad de Ciencias Exactas y Naturales \\
Universidad de Buenos Aires \&  Argentina-  {\tt  miguelsotocarlos@gmail.com}
}
\end{document}